\documentclass[11pt]{article}
\usepackage{amsmath,amsthm,amssymb}
\usepackage[margin=1in]{geometry}
\usepackage{hyperref}
\usepackage{tikz-cd}

\newtheorem{theorem}{Theorem}[section]
\newtheorem{lemma}[theorem]{Lemma}

\theoremstyle{definition}
\newtheorem{definition}[theorem]{Definition}
\newtheorem{example}[theorem]{Example}

\newcommand{\lean}[1]{}
\newcommand{\leanok}{}
\ExplSyntaxOn
\NewDocumentCommand{\uses}{m}
 {\clist_map_inline:nn{#1}{\vphantom{\ref{##1}}}%
  \ignorespaces}
\ExplSyntaxOff

\title{\textbf{A Linear Representation for Functions on Finite Sets}}

\author{
  Roman Bacik \\
  \small Vancouver, Canada \\
}
\date{\today}

\begin{document}

\maketitle

\begin{abstract}
We demonstrate that any function $f$ from a finite set $Y$ to itself can be represented linearly. 
Specifically, we prove the existence of an injective map $j$ from $Y$ into a modular ring 
$\mathbb{Z}/m\mathbb{Z}$ and a constant $a \in \mathbb{Z}/m\mathbb{Z}$ such that 
$j(f(y)) = a \cdot j(y)$ in $\mathbb{Z}/m\mathbb{Z}$ holds for all $y \in Y$. 
This result is established by analyzing the algebraic properties of the adjugate of 
the characteristic matrix associated with the function's digraph. The proof is constructive, 
providing a method for finding the embedding $j$, the modulus $m$, and the linear multiplier $a$.
\end{abstract}


\section{Introduction}

Functions on finite sets are fundamental objects in discrete mathematics, combinatorics, 
and theoretical computer science. The structure of such a function $f: Y \to Y$ is completely 
described by its functional graph (or functional digraph), a directed graph where a unique edge 
emanates from each vertex $y$ to its image $f(y)$. While these functions can exhibit complex and 
seemingly chaotic behavior, a central goal in mathematics is to find alternative representations 
that reveal underlying algebraic structure.

This article establishes that any such function becomes linear after a suitable embedding. 
We show that for any function $f$ on a finite set, one can find an injective map $j$ 
from the set into a ring of integers modulo $m$ such that the action of $f$ corresponds to 
multiplication by a constant $a$. That is for all $y \in Y$:
$j(f(y)) = a \cdot j(y)$ in $\mathbb{Z}/m\mathbb{Z}$ and hence the following diagram commutes:

\[
\begin{tikzcd}
Y \arrow[r,"f"] \arrow[d,"j"'] & Y \arrow[d,"j"] \\
\mathbb{Z}/m\mathbb{Z} \arrow[r,"a."] & \mathbb{Z}/m\mathbb{Z}
\end{tikzcd}
\]

This result bridges the gap between arbitrary discrete maps and structured linear systems. 
Without loss of generality, we can identify elements of the finite set $Y$ of size $n$ with the 
set $\{0, 1, \ldots, n-1\}$, which are elements of $\mathbb{Z}/n\mathbb{Z}$.
The full formal proof of this result is in \cite{rbacik2025}. 

\section{Main Results}

\begin{definition}[Adjacency Matrix of a Function]
\label{def:func_matrix}
\lean{func_matrix}
\leanok
Let $f: \mathbb{Z}/n\mathbb{Z} \to \mathbb{Z}/n\mathbb{Z}$ be any function. The function $f$ is represented by an 
$n \times n$ adjacency matrix $A=A_f$, where the entry 
$a_{ij} = \delta_{f(i),j}$ 
and $\delta_{i,j}$ is the Kronecker delta.
With this convention, each row of $A$ contains exactly one non-zero entry.
\end{definition}

\begin{lemma}
\label{lem:func_matrix_eq}
\lean{func_matrix_eq}
\uses{def:func_matrix}
\leanok
Let $f: \mathbb{Z}/n\mathbb{Z} \to \mathbb{Z}/n\mathbb{Z}$ be any function and $A=A_f$ be the adjacency matrix of the function $f$. 
Then for all $i\in \{0,1,\ldots,n-1\}$ and $y\in \mathbb{Z}^n$
$$(A y)_i = y_{f(i)}.$$
\end{lemma}

\begin{proof}
\leanok
The proof follows from the Definition \ref{def:func_matrix}.
$$(A y)_i = \sum_{j=0}^{n-1} a_{ij} y_j = \sum_{j=0}^{n-1} \delta_{f(i),j} y_j = y_{f(i)}$$
\end{proof}

\begin{lemma}
\label{lem:adj_eq}
\lean{adj_eq}
\uses{def:func_matrix,lem:func_matrix_eq}
\leanok
Let $f: \mathbb{Z}/n\mathbb{Z} \to \mathbb{Z}/n\mathbb{Z}$ be any function and $A=A_f$ be the adjacency matrix of the function $f$. 
Let $v\in \mathbb{Z}^n$, $y = \operatorname{adj}(x I - A) v$ and $m = \det(x I - A)$. Then for all $i\in \{0,1,\ldots,n-1\}$
$$y_{f(i)} = x y_i - m v_i.$$
\end{lemma}

\begin{proof}
\leanok
For adjugate matrix we have identity 
$(x I - A) \operatorname{adj} (x I - A) = \det(x I - A) I$.
Therefore,
$$m v = \det(x I - A) v = (x I - A) \operatorname{adj} (x I - A) v 
= (x I - A) y = x y - A y.$$
The final equality follows from the Lemma \ref{lem:func_matrix_eq}.
\end{proof}

\begin{lemma}
\label{lem:det_degree_le_sum_degrees}
\lean{det_degree_le_sum_degrees}
\leanok
Let $M$ be an $n \times n$ matrix with polynomial entries $m_{ij} \in \mathbb{Z}[x]$.
Then 
$$\deg(\det(M)) \leq \sum_{i=0}^{n-1} \sum_{j=0}^{n-1} \deg(m_{ij}).$$
\end{lemma}

\begin{proof}
\leanok
The determinant is a sum over permutations $\sigma$ of products $\prod_{i=0}^{n-1} m_{\sigma(i),i}$.
Each product has degree at most $\sum_{i=0}^{n-1} \deg(m_{\sigma(i),i})$.
Since a single element of a sum is at most the whole sum (when all terms are non-negative), 
this is bounded by $\sum_{i=0}^{n-1} \sum_{j=0}^{n-1} \deg(m_{ij})$.
The degree of a sum is at most the maximum degree of its summands.
\end{proof}

\begin{lemma}
\label{lem:charmatrix_helpers}
\lean{charmatrix_det_eq_charpoly, charmatrix_det_monic, charmatrix_det_natDegree}
\leanok
Let $A$ be an $n \times n$ matrix with integer entries. The characteristic matrix $\chi_A(x) = x I - A$ has determinant equal to the characteristic polynomial:
$$\det(\chi_A(x)) = \det(x I - A)$$
For all $n \in N$, this polynomial is monic of degree $n$.
\end{lemma}

\begin{proof}
\leanok
This follows from the standard properties of the characteristic polynomial.
\end{proof}

\begin{lemma}
\label{lem:adj_offdiag_sum_degrees_bound}
\lean{adj_offdiag_sum_degrees_bound}
\leanok
\uses{lem:det_degree_le_sum_degrees}
Let $A$ be an $n \times n$ matrix with integer entries and $\chi_A(x) = x I - A$ be its characteristic matrix.
For $i \neq j$, the $(i,j)$ entry of $\operatorname{adj}(\chi_A(x))$ has degree at most $n-2$.
\end{lemma}

\begin{proof}
\leanok
The adjugate entry $\operatorname{adj}(\chi_A(x))_{ij}$ equals the determinant of the characteristic matrix with row $j$ and column $i$ removed from $\chi_A(x)$.

This submatrix has diagonal entries from $\chi_A(x)$ except at diagonal positions $i$ and $j$ (which are deleted). 
Since $\chi_A(x)$ has diagonal entries of degree $1$ (from $x I$) and off-diagonal entries of degree $0$ (from $-A$), 
the submatrix has exactly $n-2$ diagonal entries of degree $1$ and all other entries of degree $0$.

By Lemma \ref{lem:det_degree_le_sum_degrees}, the determinant has degree at most $n-2$.
\end{proof}

\begin{lemma}
\label{lem:adj_poly}
\lean{adj_poly}
\leanok
\uses{lem:adj_offdiag_sum_degrees_bound,lem:charmatrix_helpers}
Let $M = M(x) = \operatorname{adj}(x I - A)$ be the adjugate of the characteristic matrix $x I - A$.
Then the matrix entries $m_{ij} = p_{ij}(x)$ are polynomials in $x$ for all $i,j\in \{0,1,\ldots,n-1\}$ such that
\begin{itemize}
\item $p_{ii}(x)$ is monic of degree $n-1$ for all $i\in \{0,1,\ldots,n-1\}$ and
\item $p_{ij}(x)$ has degree at most $n-2$ for all $i\neq j\in \{0,1,\ldots,n-1\}$.
\end{itemize}
\end{lemma}

\begin{proof}
\leanok
The diagonal entries of $\operatorname{adj}(x I - A)$ are characteristic polynomials of $(n-1) \times (n-1)$ submatrices, hence monic of degree $n-1$ by Lemma \ref{lem:charmatrix_helpers}. 

The off-diagonal case follows directly from Lemma \ref{lem:adj_offdiag_sum_degrees_bound}.
\end{proof}

\begin{definition}
\label{def:coeff_bound}
\lean{coeff_bound}
\leanok
For a polynomial $p(x) = \sum_{i=0}^{d} p_i x^i \in \mathbb{Z}[x]$, we define the coefficients bound:
$$|p| = \sum_{i=0}^{d} |p_i|$$
\end{definition}

\begin{lemma}
\label{lem:polynomial_positive}
\lean{polynomial_positive}
\uses{def:coeff_bound}
\leanok
If $p \in \mathbb{Z}[x]$ has positive leading coefficient, then for all integers $n \ge |p|$, we have $n > 0$ and $p(n) > 0$.
\end{lemma}

\begin{proof}
\leanok
Lemma is trivially true for $d=0$ so we can assume $d \geq 1$. Write $p(n) = an^d + r(n)$ where $a \ge 1$ is the leading coefficient of $p$, 
$d = \deg(p)$, and $\deg(r) < d$. 
For $n \ge |p|$:
$$n \ge |p| \ge a \ge 1 > 0.$$
Since $n \geq 1$, we have $|r(n)| \leq B n^{d-1}$ where $B = |p| - a$. Therefore,
$$p(n) = an^d + r(n) \geq an^d - Bn^{d-1} = n^{d-1}(an - B) \ge (a|p| - B) \ge |p| - B = a \ge 1 > 0.$$
\end{proof}

\begin{lemma}
\label{lem:adj_poly_strict_increasing}
\lean{adj_poly_strict_increasing}
\leanok
\uses{lem:adj_poly,lem:polynomial_positive,lem:charmatrix_helpers}
Let $M = (m_{ij}) = M(x) = \operatorname{adj}(x I - A)$ be the adjugate of the characteristic matrix $x I - A$.
Let $v = (1,2,\ldots,n)^T$ and $m = m(x) = \det(x I - A)$. Then for sufficiently large integer $x$:
$$0 < y_0 < y_1 < \cdots < y_{n-1} < m$$
\end{lemma}

\begin{proof}
\leanok
The proof follows from Lemma \ref{lem:adj_poly} and Lemma \ref{lem:polynomial_positive}.
Let $y = Mv$. Then $y_i = \sum_{k=0}^{n-1} m_{ik} (k+1)$.

For each entry $y_i$, we express it as evaluation of a polynomial $p_i(x) = \sum_{k=0}^{n-1} m_{ik}(x) (k+1) \in \mathbb{Z}[x]$.
By Lemma \ref{lem:adj_poly}, the diagonal entry $m_{ii}$ is monic of degree $n-1$, while off-diagonal entries $m_{ik}$ (for $k \neq i$) have degree at most $n-2$.
Therefore, the coefficient of $x^{n-1}$ in $p_i$ is $i+1 > 0$ (dominated by the $m_{ii} (i+1)$ term).

For the difference $p_j - p_i$ with $j > i$, the leading term comes from $(m_{jj} (j+1) - m_{ii} (i+1))$. Since both $m_{jj}$ and $m_{ii}$ are monic of degree $n-1$, the leading coefficient of $p_j - p_i$ is $(j+1) - (i+1) = j - i > 0$.

Similarly, for $p_m(x) = \det(x I - A) - p_i(x)$, since $\det(x I - A)$ is monic of degree $n$ (by Lemma \ref{lem:charmatrix_helpers}) and $p_i$ has degree at most $n-1$, the leading coefficient is 1.

Since $p_0(x)$ has leading coefficient $0+1 = 1 > 0$, we have $y_0 > 0$ for sufficiently large $x$.

Applying Lemma \ref{lem:polynomial_positive} to these polynomials with positive leading coefficients gives the existence of $x_0$ such that all required inequalities hold for $x > x_0$.
\end{proof}

\begin{definition}[Linear Representation]
\label{def:linear_representation}
\lean{has_linear_representation}
\leanok
Let $f: \mathbb{Z}/n\mathbb{Z} \to \mathbb{Z}/n\mathbb{Z}$ be any function. A linear representation of $f$ is an injective function $j: \mathbb{Z}/n\mathbb{Z} \to \mathbb{Z}/m\mathbb{Z}$ 
such that for all $i\in \{0,1,\ldots,n-1\}$,
$$j(f(i)) = a \cdot j(i)$$
in $\mathbb{Z}/m\mathbb{Z}$, where $m$ is a positive integer and $a$ is a multiplier from $\mathbb{Z}/m\mathbb{Z}$.
\end{definition}

\begin{lemma}[Linear Representation Lemma]
\label{lem:linear_representation_lemma}
\lean{linear_representation_lemma}
\leanok
\uses{def:linear_representation,lem:adj_eq,lem:adj_poly_strict_increasing}
For any function $f: \mathbb{Z}/n\mathbb{Z} \to \mathbb{Z}/n\mathbb{Z}$ with $n > 1$, there exists an integer $a_f$ such that for any $a > a_f$, 
we can construct a linear representation of $f$ with multiplier $a$ and modulus $m > a$.
\end{lemma}

\begin{proof}
\leanok
Let $A = A_f$ be the adjacency matrix of $f$ and let $v = (1,2,\ldots,n)^T$. 
By Lemma \ref{lem:adj_poly_strict_increasing}, there exists $x_0$ such that for all integers $x > x_0$, 
the entries $y_i$ of $y = \operatorname{adj}(xI - A)v$ satisfy:
$$0 \leq y_0 < y_1 < \cdots < y_{n-1} < m(x)$$
where $m(x) = \det(xI - A)$ is the characteristic polynomial of $A$.

Since $n > 1$, the polynomial $m(x)$ is monic of degree $n \geq 2$. 
Therefore, $m - \text{id}$ (where $\text{id}(x) = x$) is also monic of degree $n$, with leading coefficient $1 > 0$.

By Lemma \ref{lem:polynomial_positive}, the polynomial $m - \text{id}$ is positive for all $x \geq |m - \text{id}|$.

Set $a_f = \max(x_0, |m - \text{id}|)$. For any $a > a_f$, we have:
\begin{itemize}
\item $a > x_0$, so the strict inequalities $0 \leq y_0 < y_1 < \cdots < y_{n-1} < m(a)$ hold
\item $a \geq |m - \text{id}|$, so $(m - \text{id})(a) = m(a) - a > 0$, which gives $m(a) > a$
\end{itemize}

Define:
\begin{itemize}
\item $m = m(a) = \det(aI - A)$ as the modulus (note: $m > a$ by construction)
\item $j: \mathbb{Z}/n\mathbb{Z} \to \mathbb{Z}/m\mathbb{Z}$ by $j(i) = y_i \bmod m$, where $y = \operatorname{adj}(aI - A)v$
\end{itemize}

Since $0 \leq y_i < m$ for all $i$ and the $y_i$ are strictly increasing, $j$ is injective.

By Lemma \ref{lem:adj_eq}, we have $y_{f(i)} = a \cdot y_i - m \cdot v_i$ for all $i$. 
Taking this equation modulo $m$ gives:
$$j(f(i)) \equiv a \cdot j(i) \pmod{m}$$

Therefore, $j$ is a linear representation of $f$ with modulus $m > a$ and multiplier $a \in \mathbb{Z}/m\mathbb{Z}$.
\end{proof}

\begin{theorem}[Main Theorem]
\label{thm:linear_representation}
\lean{linear_representation}
\leanok
\uses{def:linear_representation,lem:linear_representation_lemma}
Any finite function $f: \mathbb{Z}/n\mathbb{Z} \to \mathbb{Z}/n\mathbb{Z}$ has a linear representation.
\end{theorem}

\begin{proof}
\leanok
For $n > 1$, apply Lemma \ref{lem:linear_representation_lemma} to obtain a threshold $a_f$ and choose $a = a_f + 1 > a_f$. 
The lemma provides an explicit construction of a linear representation for $f$ with multiplier $a_f + 1$.

For $n = 1$, the result is trivial: there is only one element in $\mathbb{Z}/1\mathbb{Z}$ (namely 0), 
so any function satisfies $f(0) = 0$. We can use $m = 1$, the identity map $j = \operatorname{id}$, and multiplier $0$, 
giving $j(f(0)) = 0 = 0 \cdot j(0)$ in $\mathbb{Z}/1\mathbb{Z}$.
\end{proof}

\subsection*{Examples}

\begin{example}[Quadratic Function in $\mathbb{Z}/3\mathbb{Z}$]
\label{ex:quadratic_Z3}
Consider the function $f: \mathbb{Z}/3\mathbb{Z} \to \mathbb{Z}/3\mathbb{Z}$ defined by $f(x) = x^2$.
This function maps:
\begin{align*}
0 &\mapsto 0 \\
1 &\mapsto 1 \\
2 &\mapsto 4 \equiv 1 \pmod{3}
\end{align*}
Despite being a non-linear function, Theorem \ref{thm:linear_representation} guarantees that $f$ has a linear representation.

The adjacency matrix is:
$$A_f = \begin{pmatrix}
1 & 0 & 0 \\
0 & 1 & 0 \\
0 & 1 & 0
\end{pmatrix}$$

The characteristic matrix is:
$$x I - A_f = \begin{pmatrix}
x-1 & 0 & 0 \\
0 & x-1 & 0 \\
0 & -1 & x
\end{pmatrix}$$

The characteristic polynomial is:
$$m = \det(x I - A_f) = (x-1)^2 \cdot x = x^3 - 2x^2 + x$$

The adjugate matrix is:
$$\operatorname{adj}(x I - A_f) = \begin{pmatrix}
x(x-1) & 0 & 0 \\
0 & x(x-1) & 0 \\
0 & (x-1) & (x-1)^2
\end{pmatrix}$$

Using vector $v = (1, 2, 3)^T$, we get:
$$y = \operatorname{adj}(x I - A_f) \cdot v = \begin{pmatrix}
x(x-1) \cdot 1 \\
x(x-1) \cdot 2 \\
(x-1) \cdot 2 + (x-1)^2 \cdot 3
\end{pmatrix} = \begin{pmatrix}
x^2 - x \\
2x^2 - 2x \\
(x-1)(3x-1)
\end{pmatrix} = \begin{pmatrix}
x^2 - x \\
2x^2 - 2x \\
3x^2 - 4x + 1
\end{pmatrix}$$

For $x = 4$, we compute:
\begin{align*}
y_0 &= 4^2 - 4 = 16 - 4 = 12 \\
y_1 &= 2(4^2) - 2(4) = 32 - 8 = 24 \\
y_2 &= 3(4^2) - 4(4) + 1 = 48 - 16 + 1 = 33 \\
m &= 4^3 - 2(4^2) + 4 = 64 - 32 + 4 = 36
\end{align*}

The injection $j: \mathbb{Z}/3\mathbb{Z} \to \mathbb{Z}/36\mathbb{Z}$ is defined by $j(i) = y_i$:
$$j(0) = 12, \quad j(1) = 24, \quad j(2) = 33$$

These values are strictly increasing and bounded by $m = 36$, so $j$ is injective.

We verify the linear representation property using Lemma \ref{lem:adj_eq}. 

The lemma states that $y_{f(i)} = x y_i - m \cdot v_i$, which we can rewrite as:
$$j(f(i)) \equiv x j(i) \pmod{m}.$$

Verification:
\begin{align*}
j(f(0)) = j(0) = 12 &\equiv 4 \cdot 12 - 36 \cdot 1 = 48 - 36 = 12 = 4 \cdot j(0) \pmod{36} \quad\checkmark \\
j(f(1)) = j(1) = 24 &\equiv 4 \cdot 24 - 36 \cdot 2 = 96 - 72 = 24 = 4 \cdot j(1) \pmod{36} \quad\checkmark \\
j(f(2)) = j(1) = 24 &\equiv 4 \cdot 33 - 36 \cdot 3 = 132 - 108 = 24 = 4 \cdot j(2) \pmod{36} \quad\checkmark
\end{align*}

Thus $j(f(i)) \equiv 4 \cdot j(i) \pmod{36}$ for all $i \in \mathbb{Z}/3\mathbb{Z}$, confirming the quadratic function $f(x) = x^2$ has a linear representation with modulus $m = 36$ and multiplier $a = 4$.
\end{example}

\section{Discussion and Context}

The representation of arbitrary functions in structured algebraic forms is a recurring theme. Our result provides one such representation, which can be compared to others.
\begin{itemize}
    \item \textbf{Polynomial Interpolation:} Any function over a finite field or ring can be represented by a polynomial (e.g., via Lagrange interpolation). Our result is distinct in that it achieves a representation as a simple linear map, $z \mapsto az$, at the cost of requiring an embedding $j$ into the ring rather than working on the original set's labels directly.
    \item \textbf{Graph Theory:} In Graph Theory our results shows that any functional digraph is a subgraph of a functional graph of a linear map.
\end{itemize}
While the modulus $m$ can be large, making this construction primarily of theoretical interest, it provides the powerful knowledge that every finite function possesses a hidden linear structure.


\end{document}